\documentclass[12pt]{article}
\usepackage{geometry}
\usepackage{float}
\usepackage{latexsym,bm}
\usepackage{amsmath,amsfonts,amsthm,amssymb,bbding,mathrsfs}
\usepackage{graphicx}

\newtheorem{theorem}{Theorem}[section]

\newtheorem{lemma}{Lemma}[section]

\newtheorem{claim}{Claim}[section]
\newtheorem{definition}{Definition}[section]
\newcommand{\ex}{\mathrm{ex}}

\title{The planar Tur\'an number of double star $S_{2,4}$}
\author{Xin Xu\thanks{School of Science, North China University of Technology, Beijing, China.}
\and Jiawei Shao\footnotemark[1]}
\date{}

\begin{document}

\maketitle

\begin{abstract}
Planar Tur\'an number $\ex_{\mathcal{P}}(n,H)$ of $H$ is the maximum number of edges in an $n$-vertex planar graph which does not contain $H$ as a subgraph. Ghosh, Gy\H{o}ri, Paulos and Xiao initiated the topic of the planar Tur\'an number for double stars. In this paper, we prove that $\ex_{\mathcal{P}}(n,S_{2,4})\leq \frac{31}{14}n$ for $n\geq 1$, and show that equality holds for infinitely many integers $n$.\\
\textbf{Keywords:}  Planar Tur\'an number, Double stars, Extremal planar graphs
\end{abstract}

\maketitle
\section{Introduction}

All graphs considered in this paper are finite and simple. Let $G=(V(G),E(G))$, where $V(G)$ and $E(G)$ are the vertex set and edge set. Let $v(G)$, $e(G)$, $\delta(G)$ and $\Delta(G)$ denote number of vertices, number of edges, minimum degree and maximum degree of $G$, respectively. We use $N_{G}(v)$ to denote the set of vertices of $G$ adjacent to $v$. Let  $N_{G}[v]=N_{G}(v)\cup \{v\}$. For any subset $S\subset V(G)$, the subgraph induced on $S$ is denoted by $G[S]$. We denote by $G\backslash S$ the subgraph induced on $V(G)\backslash S$. If $S=\{v\}$, we simply write $G\backslash v$. We use $e[S,T]$ to denote the number of edges between $S$ and $T$, where $S$, $T$ are subsets of $V(G)$.

Let $H$ be a graph, and a graph is called $H$-free if it does not contain $H$ as a subgraph. The classical problem in extremal graph theory is to determine the  $\ex(n,H)$, which gives the maximum number of edges in an $H$-free graph on $n$ vertices.  In 1941, Tur\'an~\cite{turan} gave the exact value of $\ex(n,K_{r})$, where $K_{r}$ is a complete graph with $r$ vertices. 
Later in 1946, the Erd\H{o}s-Stone Theorem~\cite{erdos1946} extended this to the case for all non-bipartite graphs $H$ and showed that
$\ex(n,H)=(1-\frac{1}{\chi(H) -1})\binom{n}{2}+o(n^{2})$, where $\chi(H)$ denotes the chromatic number of $H$. This latter result has been called
the ``fundamental theorem of extremal graph theory''~\cite{bollobs2002}.

Dowden~\cite{dowden2016} in 2016 initiated the study of planar Tur\'an-type problems. The {\it planar  Tur\'an number} of $H$, denoted by $\ex_{\mathcal{P}}(n,H)$, is the maximum number of edges in an $H$-free planar graph on $n$ vertices. Dowden studied the planar Tur\'an number of $C_{4}$ and $C_{5}$, where $C_{k}$ is a cycle with $k$ vertices. Ghosh, Gy\H{o}ri, Martin, Paulos and Xiao~\cite{ghosh2022c6} gave the exact value for $C_{6}$. Shi, Walsh and Yu~\cite{shi2023c7}, Gy\H{o}ri, Li and Zhou\cite{győri2023c7} gave the exact value for $C_{7}$. The planar Tur\'an number of $C_{k}$ is still unknown for $k\geq 8$. Cranston, Lidick\'{y}, Liu and Shantanam~\cite{daniel2022counterexample} first gave both lower and upper bound for general cycles, Lan and Song~\cite{lan2022improved} improved the lower bound. Recently, Shi, Walsh and Yu\cite{shi2023dense} improved the upper bound, Gy\H{o}ri, Varga and Zhu~\cite{győri2023new} gave a new construction and improved the lower bound. Lan, Shi and Song~\cite{lan2019hfree} gave a sufficient condition for graphs with planar Tur\'an number of $3n-6$. We refer the interested readers to more results on paths, theta graphs and other graphs~\cite{lan2019shortpaths,lan2019thetafree,ghosh2023theta6,győri2022extremal,zhai2022,fang2023,fang2022intersecting,li2024,lan2024planar,du2021,győri2023k4c5k4c6}.

Recently, Gy\H{o}ri, Martin, Paulos and Xiao\cite{ghosh2022planar} studied the topic for double stars as the forbidden graph. A $(k,l)$-star, denoted by $S_{k,l}$, is the graph obtained from an edge $uv$, and joining end vertices with $k$ and $l$ vertices respectively. They gave the exact value for $S_{2,2}$ and $S_{2,3}$. Here, we obtain the exact value for $S_{2,4}$ by a new method.%         , which is an extension of the discharging method.

\begin{theorem}\label{thm}
    Let $G$ be an $n$-vertex $S_{2,4}$-free planar graph. Then $e(G)\leq \frac{31}{14}n$ with equality holds when $n\equiv 0\ (\text{mod } 14)$.
\end{theorem}

Based on our proof technique, we also find a new extremal construction showing sharpness of Theorem~\ref{thm}.

\section{Definitions and Preliminaries}

We give some necessary definitions and preliminary results which are needed in the proof. 
Let $G$ be a planar graph.

\begin{definition}
A $\boldsymbol{k\text{-}l}$ \textbf{edge} is an edge whose end vertices are of degree $k$ and $l$. 
A $\boldsymbol{k\text{-}l\text{-}s}$ \textbf{path} is an induced path consisting of three vertices with degree $k$, $l$ and $s$.
\end{definition}

\begin{definition}
A $\boldsymbol{k\text{-}s^{-}}$ \textbf{star} is a subgraph in $G$ with $k+1$ vertices, where there is a central vertex connecting to the other $k$ vertices, and all other $k$ vertices have degree of at most $s$. 
\end{definition}

\begin{definition}
Let $H$  be a subgraph of $G$. The  weight of $H$, denoted by $\boldsymbol{w(H)}$, is defined as
\begin{center}
    $e(H)+\frac{1}{2}(e[H,G\backslash H])$.
\end{center}
\end{definition}

Obviously, $w(H)=\frac{1}{2}\sum\limits_{v\in V(H)}d(v)$, where $d(v)$ is the number of edges incident with $v$ in $G$.

We shall make use of the following lemma in the proof of Theorem~\ref{thm}.

\begin{lemma}\label{lem}
    Let $G$ be an $n$-vertex $S_{2,4}$-free planar graph with $\delta(G)\geq 3$. Then 
    $$e(G)=w(G)\leq \frac{31}{14}n.$$
\end{lemma}
%Note that $e(G)=w(G)=\sum\limits_{i=1}^{t}w(G_{i})$. The following corollary is a direct consequence of Lemma\ref{lem}.

%%%%%%%%%%%%%%%%%%%%%%%%%%%%%%%%%%%%%%%%%%%

%\section{Proof of Lemma~\ref{lem}}

\begin{proof}
We now describe the proof strategy. The graph $G$ will be decomposed into vertex disjoint subgraphs, and no subgraph contributes too much towards the total weight.
It is shown that there exists a  vertex partition  $V(G)=\bigcup\limits_{i=1}^{t}V(G_{i})$ with $V(G_{i})\cap V(G_{j})=\emptyset$ for any $i\neq j$, such that  $w(G_{i})\leq \frac{31}{14}v(G_{i})$. Moreover, for $1\leq i\leq t-1$, $G_{i}$ is a subgraph base on a $k$-$l$ edge, a $k$-$l$-$s$ path, or a $k$-$s^{-}$ star.

We first restrict the range of vertex degree in $G$ .

\begin{claim}
    $\Delta(G)\leq 6$.
\end{claim}
\begin{proof}
    Recall that $\delta(G)\geq 3$. If there is a vertex of degree at least $7$, then there exists an $S_{2,4}$ in $G$, a contradiction.     
\end{proof}

\begin{claim}
    If there exists a vertex of degree $6$, say $v$, in $G$, then $G[N[v]]$ is a connected component. 
\end{claim}
\begin{proof}
    Let $v$ be the vertex of degree $6$ and  
    $u\in N(v)$.
    When $n=7$, it is trivial.
    Assume $n\geq 8$. If $u$ has a neighbor in the $G\backslash N[v]$, then $G$ contains an $S_{2,4}$, a contradiction.  So there is no edge between $N[v]$ and $G\backslash N[v]$, which implies $G[N[v]]$ is a connected component.    
\end{proof}

Note that $w(G[N[v]])=e(G[N[v]])\leq 15\leq \frac{31}{14}\cdot 7$.
%Hence, we consider this connected component as some $G_{i}$. Then  $w(G_{i})=e(G_{i})\leq 15\leq \frac{31}{14}\cdot 7$. 

 %It  remains to prove the cases for vertices of degree $5$. 
 Assume that $\Delta(G)\leq 5$.  We will show that for each vertex of degree $5$, there exists a subgraph $H$ containing it with $w(H)\leq \frac{31}{14}v(H)$.

 %We use $G^{\star}$ to denote the subgraph induced by the set of all $5$-degree vertices.  

\begin{figure}[ht]
  \centering  \includegraphics[width=0.85\textwidth]{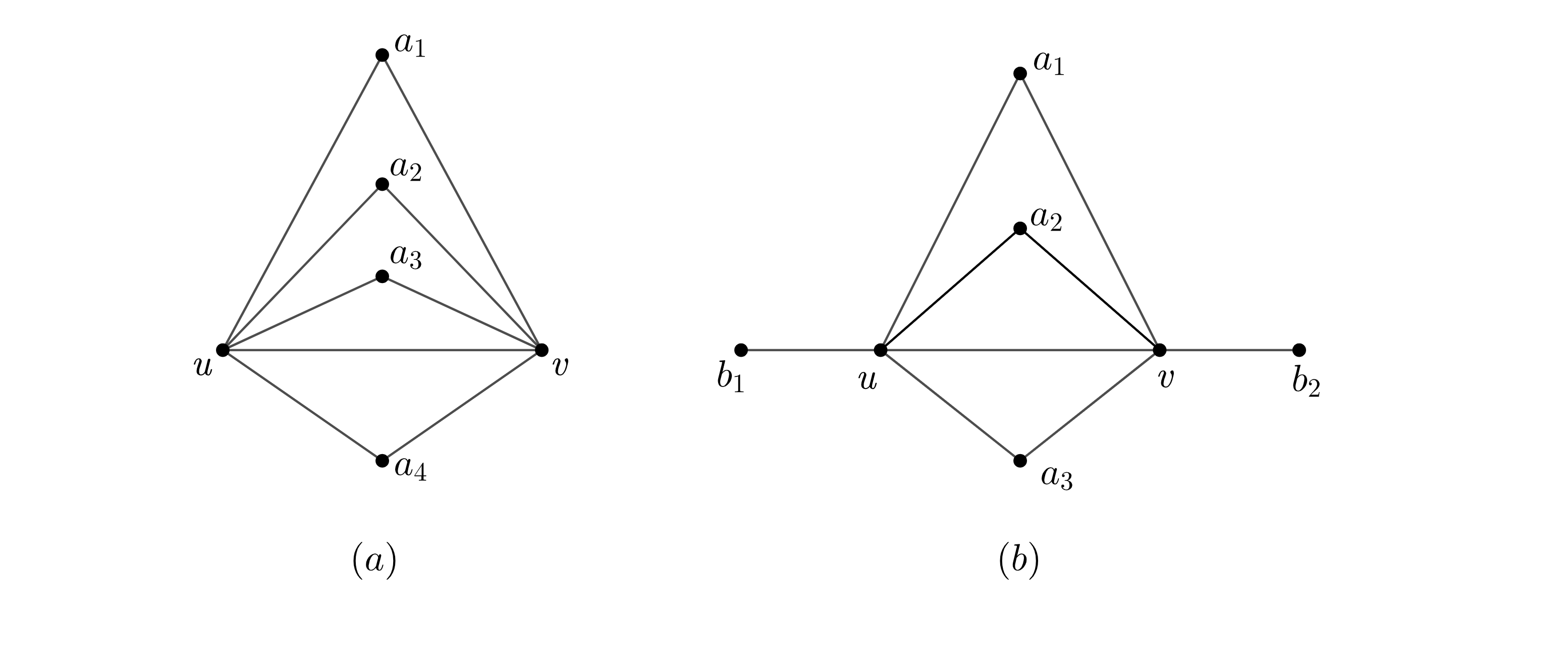}
  \caption{A $5$-$5$ edge with $4$ or $3$ triangles sitting on it.}
  \label{fig_55}
\end{figure}

\textbf{Case 1.} There exists a $5$-$5$ edge in $G$.

Let $uv$ be the $5$-$5$ edge in $G$. There exist at least $3$ triangles sitting on the  $uv$, otherwise an $S_{2,4}$ is found. Now, we distinguish the cases based on the number of triangles sitting on $uv$.

\textbf{Case 1.1.} There are $4$ triangles sitting on $uv$.

Let $a_{1}, a_{2}, a_{3}$ and $a_{4}$ be the vertices adjacent to both $u$ and $v$, as shown in Figure \ref{fig_55}$(a)$. Let $S=\{u, v, a_{1}, a_{2}, a_{3}, a_{4}\}$, $S_{1}=\{a_{1}, a_{2}, a_{3}, a_{4}\}$, $S_{2}=\{u, v\}$, $S'=V(G)\backslash S$, $H=G[S]$, $H'=G[S']$ and $H_{i}=G[S_{i}]$ for $i=1,2$. 
Hence we have $w(H)= 9+e(H_{1})+\frac{1}{2}e[H, H']$.

Note that all vertices in $S_{1}$ can form a path of length at most $3$ and each vertex in  $S_{1}$ can have at most one neighbor in $H'$, otherwise $G$ contains an $S_{2,4}$. This means that $e(H_{1})\leq 3$ and $e[H, H']\leq 4$.

If $e(H_{1})\leq 2$ or $e[H, H']\leq 2$, then $w(H)\leq 13\leq \frac{31}{14}\cdot 6$.
Assume that $e(H_{1})=3$ and $e[H, H']\geq 3$. There must exist a vertex, say $a_{1}$, of degree $5$. Let $a'_{1}\in S'$ and $a_{1}a'_{1}\in E(G)$ . 
We claim that the vertex $a'_{1}$ has the other neighbor in $H$. Otherwise an $S_{2,4}$ is contained by $d(a'_{1})\geq 3$. Moreover, it is obtained $d(a'_{1})=3$.

Let $S^{*}=S\cup a'_{1}$ and $H^{*}=G[S^{*}]$. It follows that 
\begin{align*}
    w(H^{*})&=w(H)+w(a'_{1})\\
    &=14+\frac{1}{2}e[H^{*}, G\backslash H^{*}]\\
    &\leq \frac{31}{2}=\frac{31}{14}\cdot 7.
\end{align*}

\begin{figure}[ht]
  \centering  \includegraphics[width=0.8\textwidth]{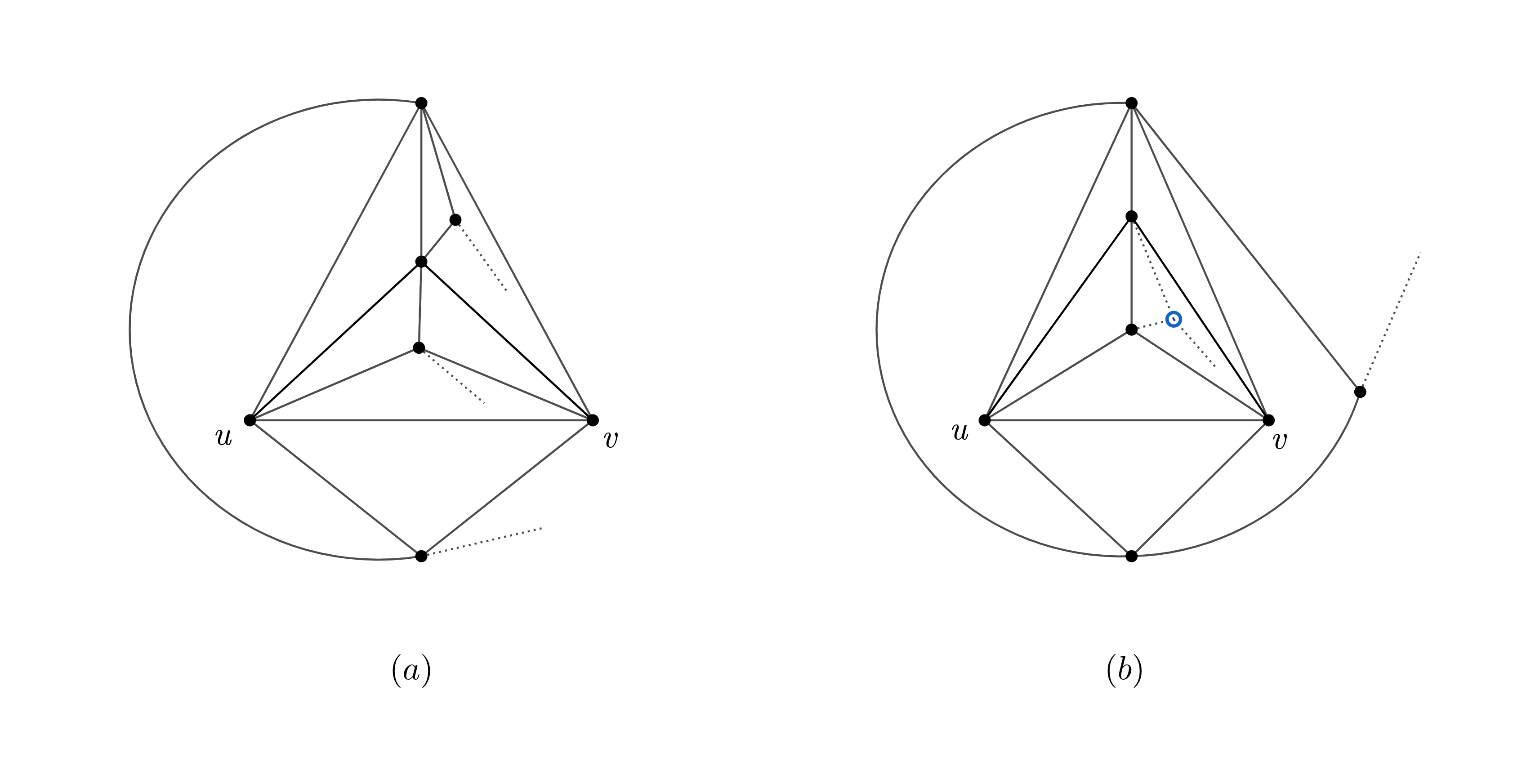}
  \caption{The subgraphs attaining the upper bound where $e(H_{1})=3$.}
  \label{fig_55_e1}
\end{figure}

In fact, there are two non-isomorphic subgraphs attaining the bound, as shown in Figure \ref{fig_55_e1}. 
%It is noticed that only the subgraph $(a)$ is a part of the extremal graph. 
Specially, for the subgraph $(b)$, we have $d(a_{2})=5$. Similarly, there must exist a vertex $a'_{2}$ such that $a'_{2}a_{2}, a'_{2}a_{3}$ are both edges in $G$. Let $S^{*}=S\cup \{a'_{1}, a'_{2}\}$ and $H^{*}=G[S^{*}]$. It follows that
$w(H^{*})=16+\frac{1}{2}e[H^{*}, G\backslash H^{*}]= 17= \lfloor \frac{31}{14}\cdot 8\rfloor$.

\textbf{Case 1.2.} There are $3$ triangles sitting on $uv$.\par
Let $a_{1}$, $a_{2}$ and $a_{3}$ be the vertices adjacent to both $u$ and $v$. Let $b_{1}$ be the vertex only adjacent to $u$ and $b_{2}$ be the vertex only adjacent to $v$, see Figure \ref{fig_55}$(b)$.  Let $S=\{u, v, a_{1}, a_{2}, a_{3}, b_{1}, b_{2}\}$, $S_{1}=\{a_{1}, a_{2}, a_{3}\}$, $S_{2}=\{b_{1}, b_{2}\}$ and $S_{3}=\{u, v\}$. That means $S = S_{1}\cup S_{2}\cup S_{3}$. And let $S'=V(G)\backslash S$, $H=G[S]$, $H'=G[S']$ and $H_{i}=G[S_{i}]$ for $i\in \{1, 2, 3\}$.

Thus we have 
$$ w(H)=9+e(H_{1})+e[H_{1}, H_{2}]+e(H_{2})+\frac{1}{2}(e[H_{1} , H']+e[H_{2} , H']).$$

Similarly, all vertices in $S_{1}$ can form a path of length at most $2$ and each vertex in  $S_{1}$ can have at most one neighbor in $H'$, which means that $e(H_{1})\leq 2$ and $e[H_{1} , H']\leq 3$. Besides, each vertex in  $S_{2}$ can have at most one neighbor in $S'\cup S_{2}$. Hence if $e(H_{2})=1$, then $e[H_{2} , H']=0$, which implies $e(H_{2})+\frac{1}{2}e[H_{2} , H']\leq 1$. If $a_{i}b_{j}\in E(G)$ for $i=1, 2, 3$ and $j=1, 2$, then $a_{i}$ can not have a neighbor in $H'$. This means that $e[H_{1}, H_{2}]+\frac{1}{2}e[H_{1} , H']\leq \frac{9}{2}$.

Moreover, we get $d(b_{1}), d(b_{2})\leq 4$, otherwise $G$ contains an $S_{2,4}$.

\begin{figure}[ht]
  \centering  \includegraphics[width=1\textwidth]{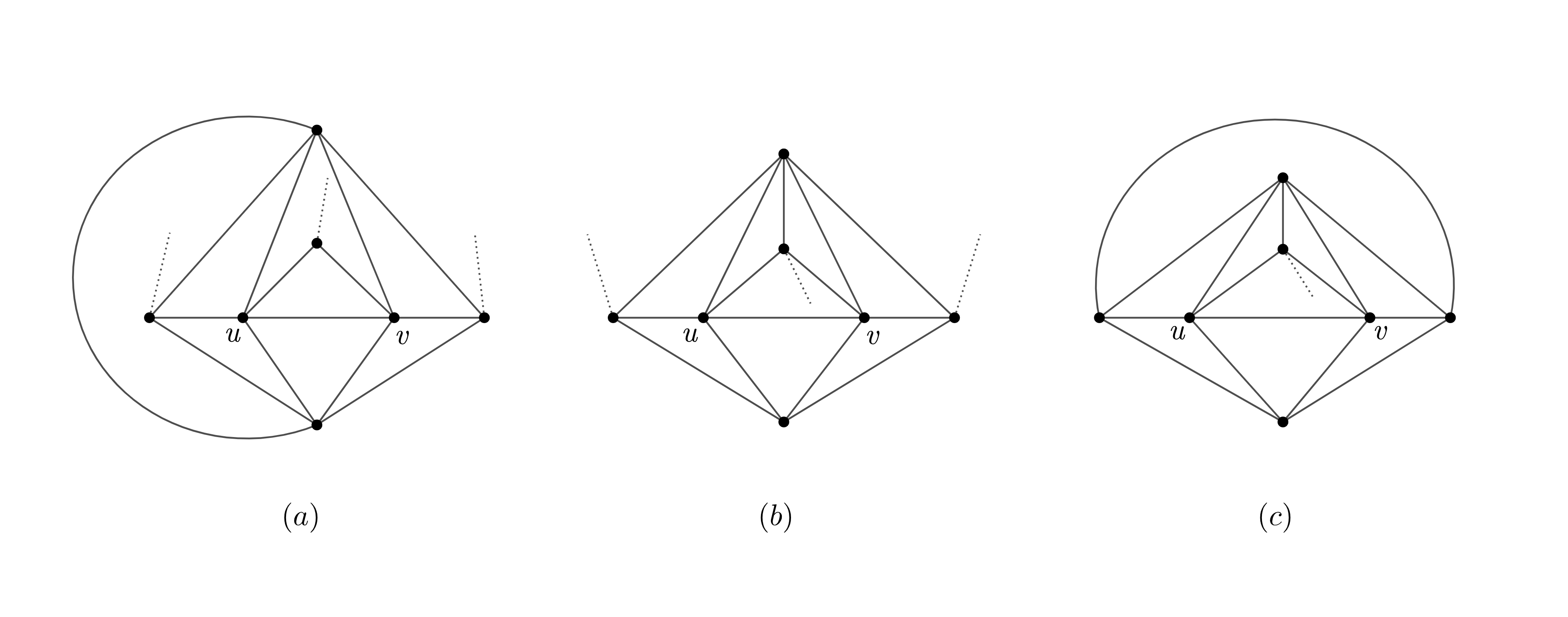}
  \caption{The subgraphs attaining the upper bound where $e(H_{1})=1$.}
  \label{fig_55_e2}
\end{figure}

When $d(b_{1})=d(b_{2})=4$, we have $b_{1}, b_{2}$ each have exactly two neighbors in $S_{1}$. Since $a_{2}, a_{3}$ are in different regions, it is concluded that $a_{1}$ must be the common neighbor of  $b_{1}, b_{2}$. 
If $e(H_{1})=2$, then $a_{1}a_{2}, a_{1}a_{3}$ are edges in $G$. We obtain that $d(a_{1})=6$, which contradicts  the assumption that $\Delta(G)\leq 5$.
So $e(H_{1})\leq 1$. Then we have 
$w(H)\leq 11+e[H_{1}, H_{2}]+\frac{1}{2}e[H_{1} , H']\leq 15\frac{1}{2}=\frac{31}{14}\cdot 7$. The possible subgraphs attaining the bound are shown in Figure~\ref{fig_55_e2}. Note that Figure~\ref{fig_55_e2}$(a)$ is isomorphic to the graph in Figure~\ref{fig_55_e1}$(a)$. Furthermore,  Figure~\ref{fig_55_e2}$(b)$  requires a more in-depth discussion.
If $b_{1}b_{2}$ is not an edge, an $S_{2,4}$ is contained in this subgraph.
Thus $b_{1}b_{2}$ is an edge, as shown in Figure~\ref{fig_55_e2}$(c)$.

Now assume that $d(b_{1})= 3, d(b_{2})=4$. 
If $|N(b_{1})\cap S_{1}|=2$, then there exists a vertex, say $a_{1}$, adjacent to $b_{1}, b_{2}$. Similarly, we have $e(H_{1})\leq 1$ and $e(H_{2})+\frac{1}{2}e[H_{2} , H']\leq \frac{1}{2}$. It follows that $w(H)\leq 15< \frac{31}{14}\cdot 7$.
Assume $|N(b_{1})\cap S_{1}|=1$. Then it is obtained that $e[H_{1}, H_{2}]+\frac{1}{2}e[H_{1} , H']\leq \frac{7}{2}$. Hence $w(H)\leq 15\frac{1}{2}=\frac{31}{14}\cdot 7$, where the equality holds when $e(H_{1})=2$. The subgraphs attaining the upper bound is shown in Figure~\ref{fig_55_e3}, which is also isomorphic to the one in Figure~\ref{fig_55_e1}$(a)$.

\begin{figure}[ht]
  \centering  \includegraphics[width=0.65\textwidth]{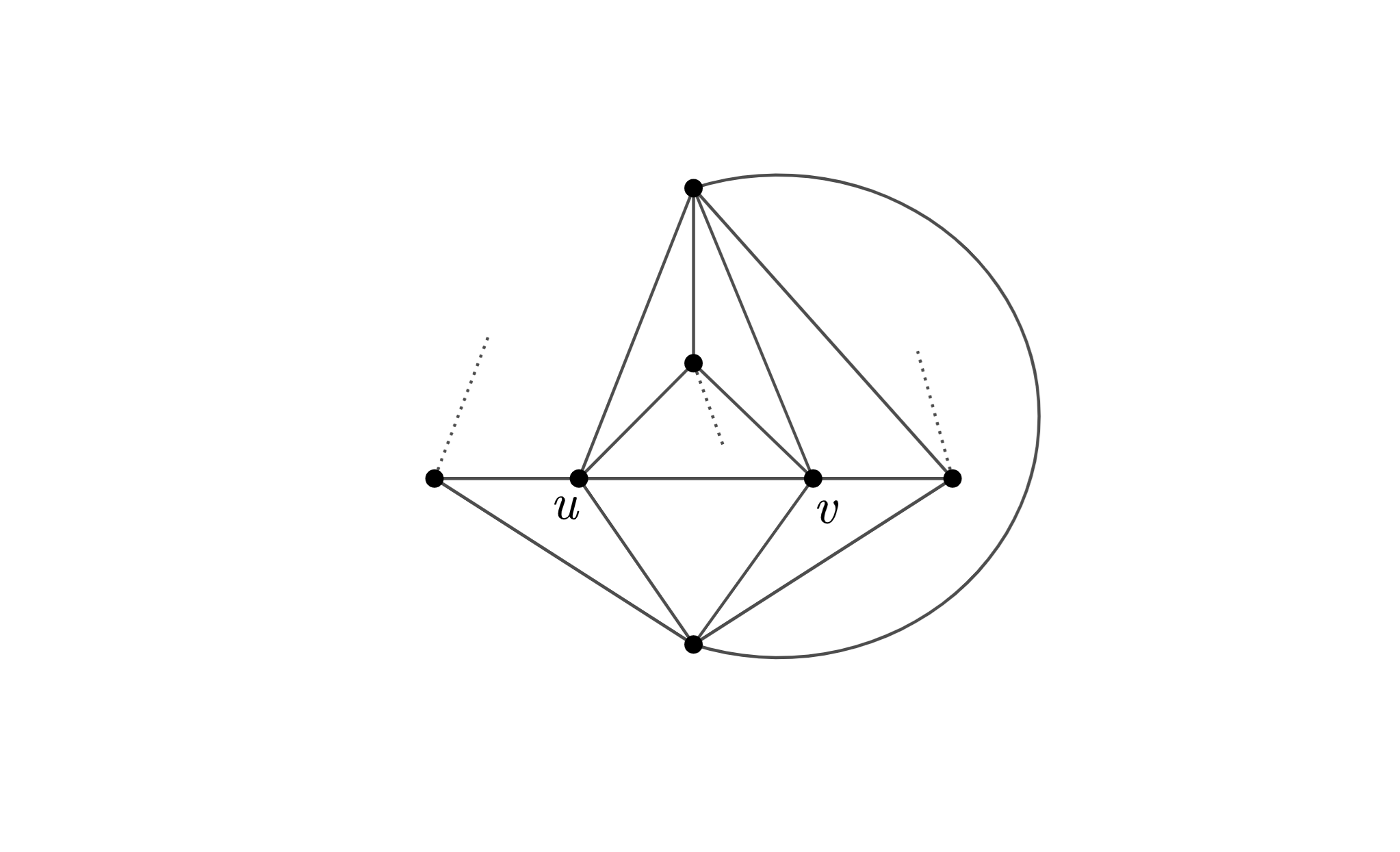}
  \caption{The subgraph attaining the upper bound where $e(H_{1})=2$.}
  \label{fig_55_e3}
\end{figure}

It remains to consider the case when  $d(b_{1})= 3, d(b_{2})= 3$. If $|N(b_{1})\cap S_{1}|=|N(b_{2})\cap S_{1}|=2$, there exists a vertex, say $a_{1}$, with degree at least $4$. This implies that $e(H_{1})\leq 1$ and  $e(H_{2})+\frac{1}{2}e[H_{2} , H']=0$. Thus $w(H)\leq 14\frac{1}{2}<\frac{31}{14}\cdot 7$.
Now we have $|N(b_{1})\cap S_{1}|\leq 2$ and $|N(b_{2})\cap S_{1}|\leq 1$. Then $e[H_{1}, H_{2}]+\frac{1}{2}e[H_{1} , H']\leq \frac{7}{2}$, with equality when $e(H_{2})=0$ and $e[H_{2} , H']\leq 1$. It follows that $w(H)\leq 15< \frac{31}{14}\cdot 7$.

%we can obtain that 
%$e[H_{1}, H_{2}]+\frac{1}{2}e[H_{1} , H']\leq \frac{7}{2}$. It follows that $w(H)\leq 15\frac{1}{2}=\frac{31}{14}\cdot 7$. The equality holds when the graph is shown in Figure \ref{}. It should be noted that this graph is isomorphic to the graph obtained in Case 1.1.

\textbf{Case 2.}  There exists a $5$-$4$-$5$ path in $G$.

%All vertices of degree $5$ form an independent set in $G$.  

%Assume that there are two $5$-degree vertices having common neighbors.

%\textbf{Case 2.1.} There exists a $5$-$4$-$5$ edge in $G$.

There are two possible planar embeddings.

\begin{figure}[ht]
  \centering  \includegraphics[width=0.9\textwidth]{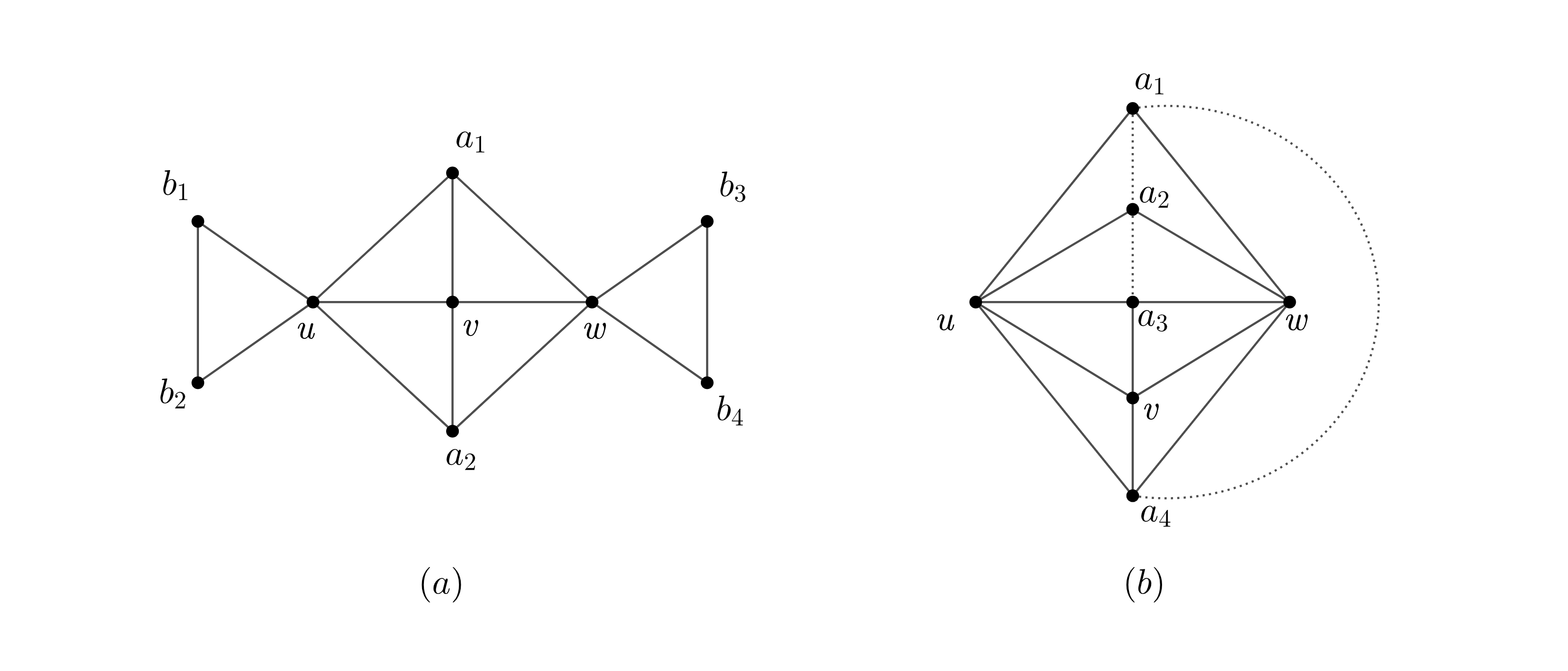}
  \caption{The subgraphs containing a $5$-$4$-$5$ path.}
  \label{fig_545}
\end{figure}

For the first planar embedding, as shown in Figure~\ref{fig_545}$(a)$, let $S=\{u, v, w, a_{1}, a_{2}, b_{1}, b_{2}, b_{3}, b_{4}\}$, $S_{1}=\{a_{1}, a_{2}\}$, $S_{2}=\{b_{1}, b_{2}, b_{3}, b_{4}\}$ and $H=G[S]$. Note that there is no edge between $S_{1}$
and $S_{2}$. Otherwise an $S_{2,4}$ is found. Since $\delta(G)\geq 3$, we get $b_{1}b_{2}, b_{3}b_{4}\in E(G)$. Note that $a_{1}$ may be adjacent to $a_{2}$ and each vertex in $S_{2}$  has exactly a neighbor outside of $H$. It follows that $w(H)\leq 15+\frac{4}{2}< \frac{31}{14}\cdot 9$.

For the second planar embedding, as shown in Figure~\ref{fig_545}$(b)$,  let $S=\{u, v, w, a_{1}, a_{2}, a_{3}, a_{4}\}$ and $H=G[S]$. It is easy to check that each vertex in $S$ have no neighbor outside of $H$. We have $H$ is a connected component. It follows $w(H)\leq 15< \frac{31}{14}\cdot 7$.

\textbf{Case 3.} There exists a $5$-$3$-$5$ path in $G$.

Without containing a $5$-$5$ edge or a $5$-$4$-$5$ path as a subgraph, $G$ has the only subgraph structure based on $5$-$3$-$5$ edge, shown in Figure \ref{fig_535}.

Let $S=\{u, v, w, v', a_{1}, a_{2}, a_{3}, b_{1}, b_{2}, b_{3}\}$, $S_{1}=\{a_{1}, a_{2}, a_{3}, b_{1}, b_{2}, b_{3}\}$ and $H=G[S]$.

\begin{figure}[ht]
  \centering  \includegraphics[width=0.7\textwidth]{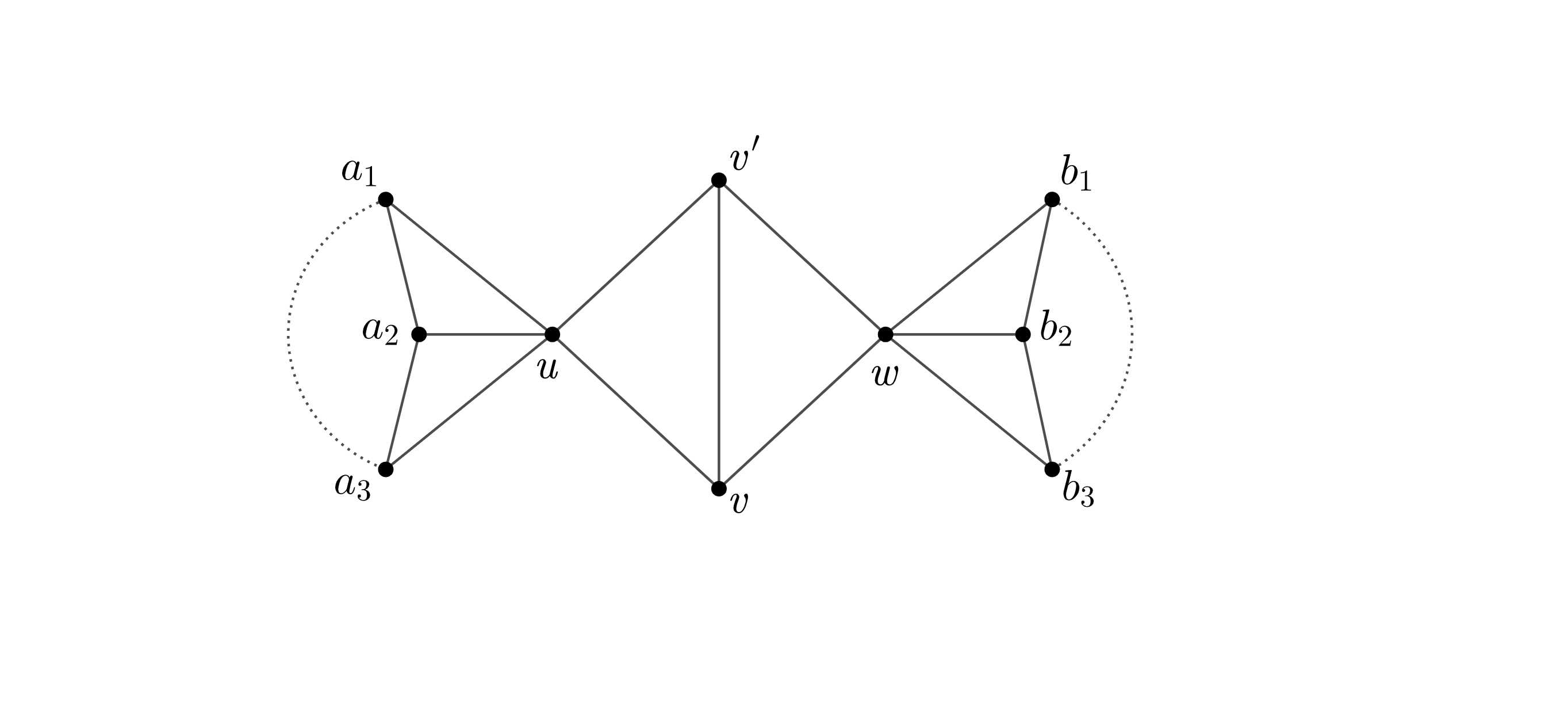}
  \caption{The subgraph containing a $5$-$3$-$5$ path.}
  \label{fig_535}
\end{figure}

Note that $a_{1}$ must be adjacent to some vertex in $\{a_{2}, a_{3}\}$, otherwise we find an $S_{2,4}$. Moreover, each vertex in $S_{1}$ has at most a neighbor outside of $H$. Hence, we have $w(H)\leq 17+\frac{6}{2}< \frac{31}{14}\cdot 10$.

\textbf{Case 3.} There exists a $5$-$4^{-}$ star in $G$.

Let $H$ be the induced graph by the $5$-$4^{-}$ star.
It is easy to know that $w(H)\leq \frac{1}{2}(5\cdot 1+4\cdot 5)\leq \frac{31}{14}\cdot 6$.

Now we study a class of special subgraphs that may be contained in $G$. Given a $5$-$4$-$5$ path shown in Figure~\ref{fig_545}$(a)$, if there exists a vertex, say $u'$, with $d(u')=5$
such that $u\neq u'$ and $u'b_{1}\in E(G)$, we have $u'b_{2}\in E(G)$. Otherwise an $S_{2,4}$ is found.
Then the subgraph containing vertices $u, v, w, u'$ must be the one shown in Figure~\ref{fig_54535}$(a)$.
It is worth noting that this subgraph can be obtained by merging a $5$-$4$-$5$ path and a $5$-$3$-$5$ path. 
We can continue this process until the vertices in the subgraph are no longer adjacent to other vertices of degree $5$. This subgraph is called the maximal expansion based on $5$-$4$-$5$ path. 
The graph in Figure~\ref{fig_54535}$(b)$ shows the structure obtained after two expansion operations. 
Similarly, the maximal expansion based on $5$-$3$-$5$ path can be defined in the same way.
It can be easily checked that the weight of the maximal expansion satisfies the bound.

\begin{figure}[ht]
  \centering  \includegraphics[width=0.95\textwidth]{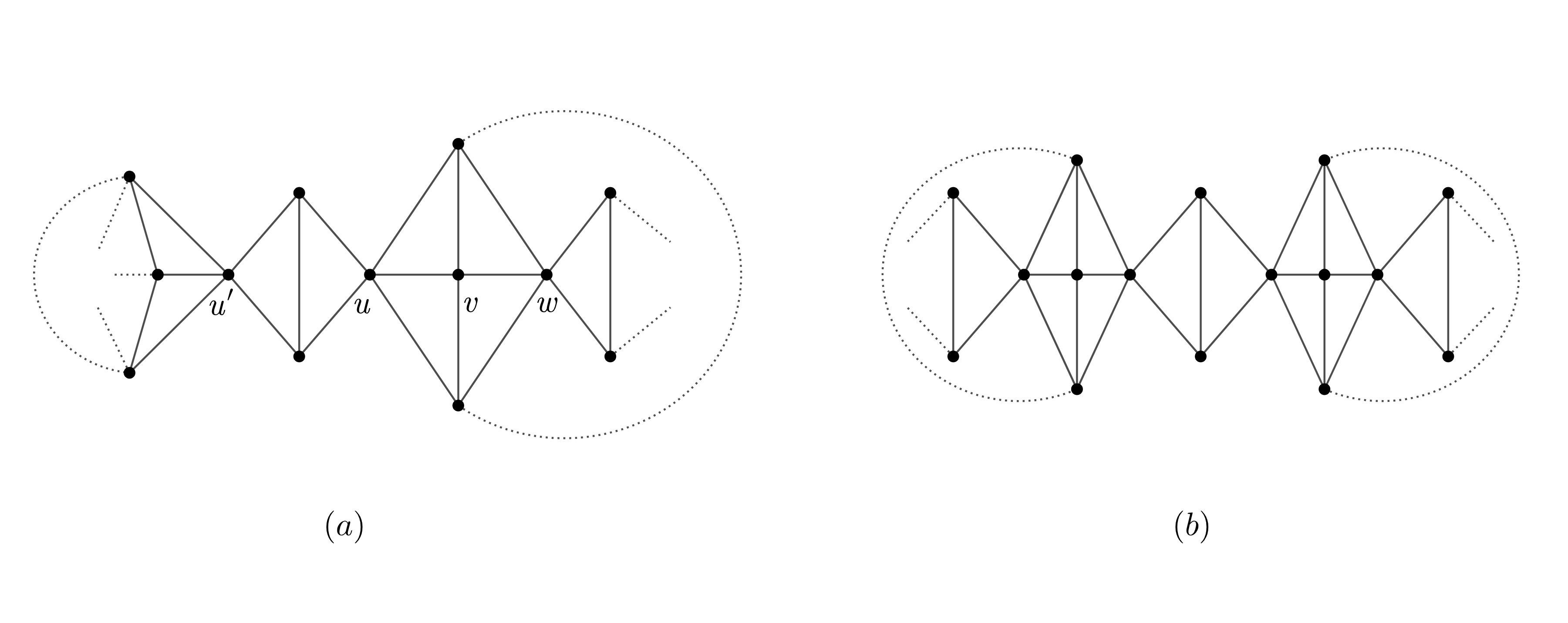}
  \caption{The expansion based on a $5$-$4$-$5$ path.}
  \label{fig_54535}
\end{figure}

Next we decompose the graph $G$ into vertex disjoint subgraphs through the following steps. 

(I) For each vertex $u$ of degree $6$, we have $G[N[u]]$ is a connected component. Let $G_{i}, i=1,\cdots, p$ be such kind of components. %We have $G^{1}=G\backslash \bigcup\limits_{i=1}^{p}G_{i}$.

(II) For each vertex of degree $5$, the vertex is inspected sequentially according to the following rules, and the first subgraph structure identified is denoted by $G_{i}$. 
Assume $u$ be the vertex.

(a) There exists a vertex $v$ of degree $5$ such that $uv\in E(G)$.  Then the subgraph based on this $5$-$5$ edge we discuss above is considered as $G_{i}$. 

(b) There exists a vertex $v$ of degree $5$ such that  $N(u)\cap N(v)\neq \emptyset$.  (i) If  $u, v$ are contained in  a $5$-$4$-$5$ path, then the maximal expansion based on this $5$-$4$-$5$ path is the $G_{i}$.
(ii) If  $u, v$ are contained in  a $5$-$3$-$5$ path, then the maximal expansion  based on this $5$-$3$-$5$ path is the $G_{i}$.

Specially, the vertex $v$ is determined too.

(c) The  vertex  $u$ is contained in a $5$-$4^{-}$ star. Then  let $G_{i}$ be the $5$-$4^{-}$ star.

Repeat the process until there is no vertex of degree $5$.
Let $G_{i}, i=p+1,\cdots, p+q$ be such subgraphs. 

(III) All vertices left are degree of at most $4$. Let $G_{p+q+1}$ be the graph induced by all these vertices.

We show that different subgraphs obtained here are vertex disjoint. If $G_{i}$ contains a vertex of degree $6$, then $G_{i}$ is a connected component on $7$ vertices. 
If $G_{i}$ is a subgraph based on a $5$-$5$ edge, $5$-$4$-$5$ path or $5$-$3$-$5$ path, the vertex set of $G_{i}$ consists of some vertices of degree $5$ and their neighbors. 
Moreover, any vertex in $G_{i}$ which has a neighbor outside can not be adjacent to some vertex of degree $5$ in $G\backslash G_{i}$. Otherwise, an $S_{2,4}$ is found. If $G_{i}$ is a $5$-$4^{-}$ star, any vertex in $G_{i}$ can not be adjacent to some vertex of degree $5$ in $G\backslash G_{i}$, otherwise it contradicts with the decomposition rules above.

Hence we construct a vertex partition 
\begin{align*}
    V(G)&=\bigcup\limits_{i=1}^{p}V(G_{i})\cup \bigcup\limits_{i=p+1}^{p+q}V(G_{i})\cup V(G_{p+q+1})\\    &=\bigcup\limits_{i=1}^{t}V(G_{i})
\end{align*}
where $t=p+q+1$.

For $1\leq i\leq p+q$, we have  $w(G_{i})\leq \frac{31}{14}v(G_{i})$ by the discussion above.
For $i=t$, it is obtained  $w(G_{t})=\frac{1}{2}\sum\limits_{v\in V(G_{t})}d(v)\leq 2v(G_{t})\leq \frac{31}{14}v(G_{t})$.

Therefore,    $e(G)=w(G)=\sum\limits_{i=1}^{t}w(G_{i})\leq \sum\limits_{i=1}^{t}\frac{31}{14}v(G_{i})=\frac{31}{14}v(G)$. The lemma is proved.

%%%%%%%%%%%%%%%%%%%%
% 不需要这个内容
%\textbf{Case 2.2.} All vertices of degree of $5$ have no common neighbor.

%It is obtained that $n_{5}\leq \frac{n}{6}$.
%Then 
%\begin{align*}
%    e(G)&=\sum\limits_{v\in V(G)}d(v)\\
%    &=\frac{1}{2}(5n_{5}+4n_{4}+3n_{3})\\
%    &\leq \frac{1}{2}[5n_{5}+4(n-n_{5})]\\
%    &\leq \frac{25}{12}n\\
%    &\leq \frac{31}{14}n.
%\end{align*}
%%%%%%%%%%%%%%%%%%%%%%%%%%%%%%%%%%%

\end{proof}

\section{Planar Tur\'an number of $S_{2,4}$}

Here we give the proof of the Theorem~\ref{thm} and construct the extremal planar graphs.

\begin{proof}
    Assume that $\mathcal{G}$ is the set of all $S_{2,4}$-free planar graphs.
    Then for each graph $G\in \mathcal{G}$, we define the operation
    \begin{itemize}
        \item Delete the vertex of degree at most $2$.
    \end{itemize}
    Repeat the operation until it can no longer go on. The induced graph is denoted by $G'$. We know that $\delta(G')\geq 3$ or $G'$ is an empty graph.

    If $\delta(G')\geq 3$, it is obtained  that  $e(G')\leq \frac{31}{14}v(G')$ by Lemma~\ref{lem}.

    Hence,
    \begin{align*}
        e(G)&\leq e(G')+2(v(G)-v(G'))\\
        &\leq \frac{31}{14}v(G')+2(v(G)-v(G'))\\
        &\leq \frac{31}{14}v(G).
    \end{align*}
    If $G'$ is an empty graph, then $e(G)\leq e(G')+2(v(G)-v(G'))\leq 2v(G)\leq \frac{31}{14}v(G)$.

    This completes the proof.
\end{proof}

Now we complete it by demonstrating that this bound is tight. 
It is known that the number of edges in extremal graphs must attain the equality in Theorem~\ref{thm}. This means $V(G)$ has a vertex partition $\bigcup\limits_{i=1}^{t}V(G_{i})$ such that each $G_{i}$ is some subgraph based on a $5$-$5$ edge. By Lemma~\ref{lem}, there are exactly two non-isomorphic $7$-vertex subgraphs attaining the bound $\frac{31}{14}\cdot 7$ and one $8$-vertex subgraph attaining the bound $\lfloor \frac{31}{14}\cdot 8\rfloor$.

Now we liken the construction of extremal planar graph to building with "blocks".
Each subgraph in Figure~\ref{fig_55_e2} is treated as a building block, and by piecing together different  blocks, we can obtain the extremal graph.
Let $H^{\#}, H^{*}$ be the corresponding subgraphs in Figure~\ref{fig_55_e2} $(a, c)$.
Note that $e(H^{\#})=14$ and $e(H^{*})=15$.

\begin{figure}[ht]
  \centering  \includegraphics[width=0.6\textwidth]{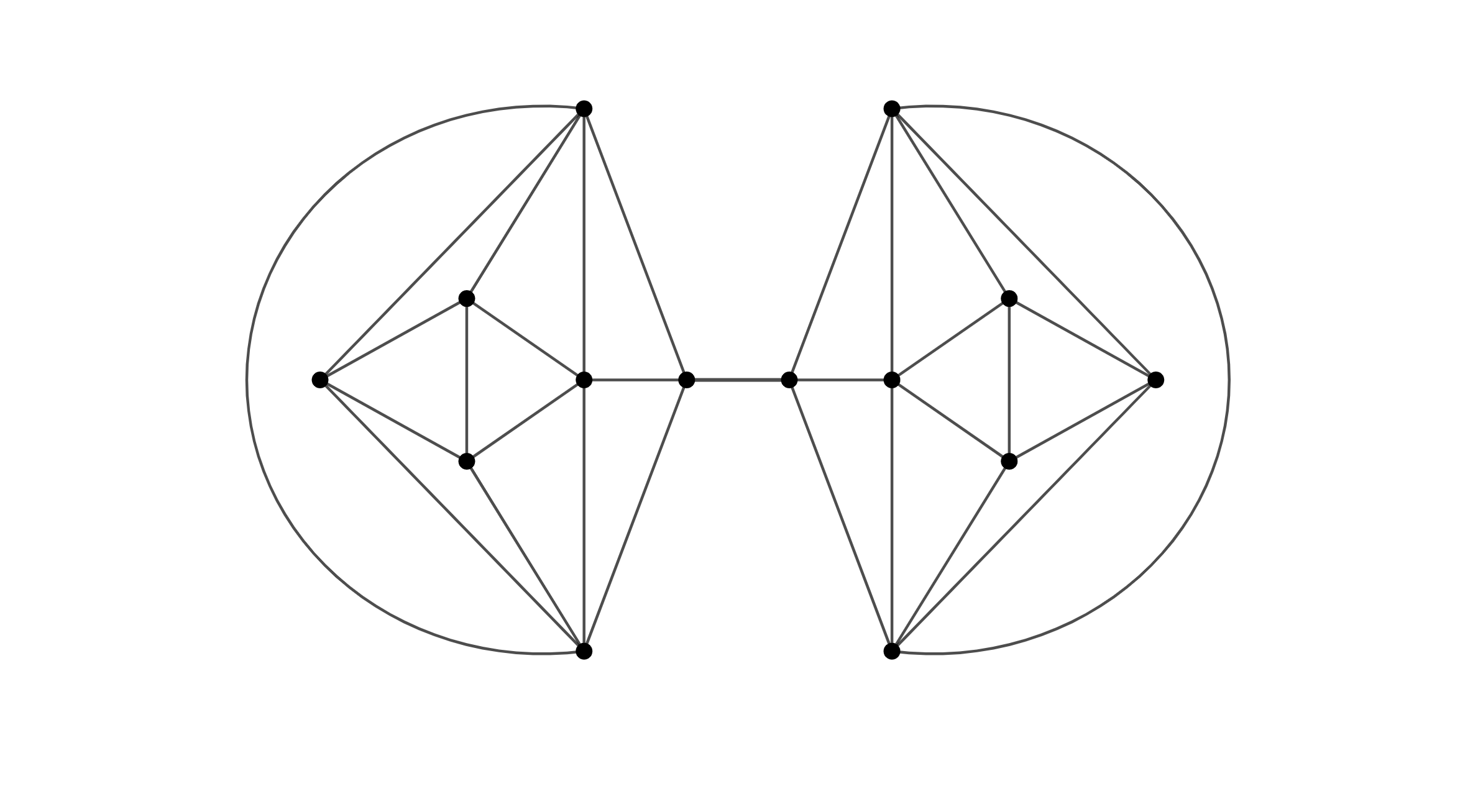}
  \caption{The extremal planar graph with $14$ vertices and $31$ edges.}
  \label{fig_ex}
\end{figure}
For example,  an extremal planar graph with $14$ vertices and $31$ edges is constructed by two $H^{*}$'s, as shown in Figure~\ref{fig_ex}.

We redraw these two subgraphs $H^{\#}, H^{*}$ by contracting $H^{\#}$ or $H^{*}$ into a single vertex and keeping the edges incident with it,  as shown in Figure~\ref{fig_ex2}$(a)$. Then the graph in Figure~\ref{fig_ex} is illustrated as  Figure~\ref{fig_ex2}$(b)$. 
\begin{figure}[ht]
  \centering  \includegraphics[width=0.95\textwidth]{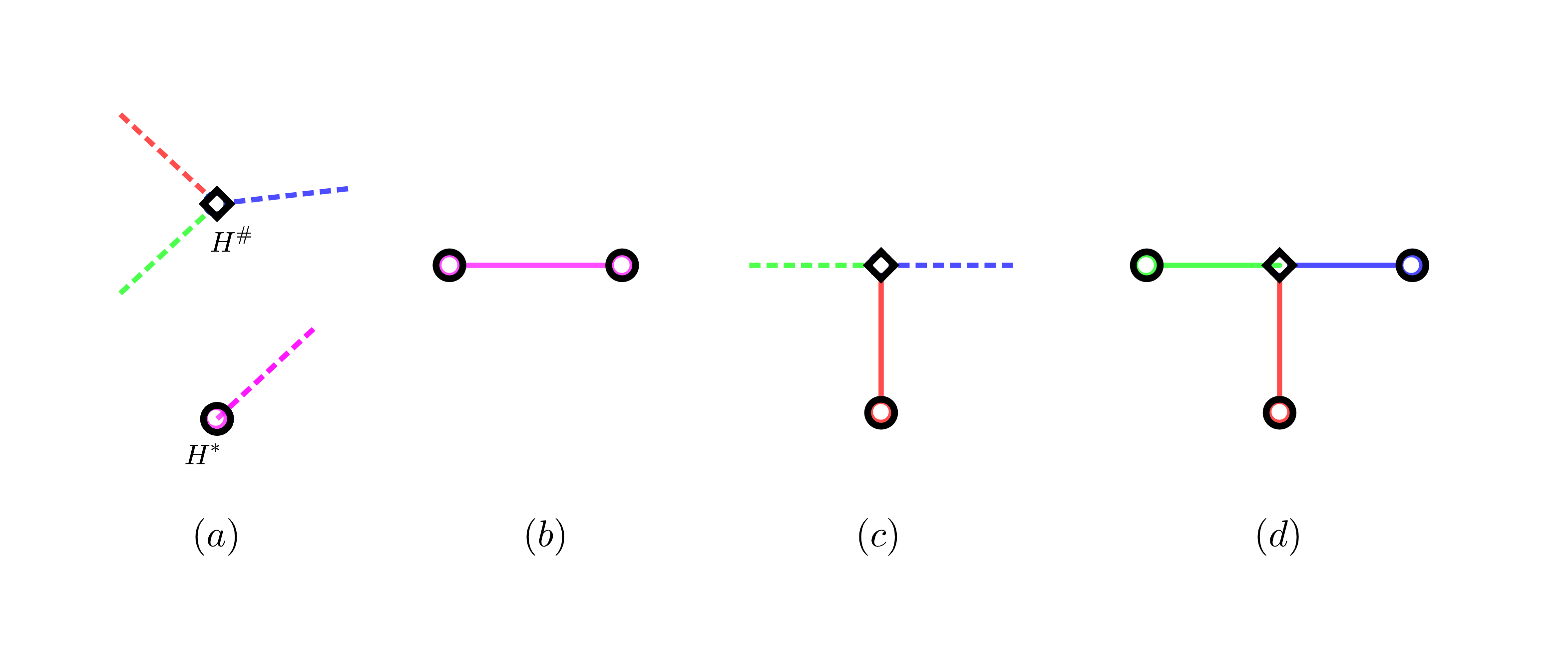}
  \caption{$(a), (c)$ the building blocks; $(b)$, $(d)$ some extremal graphs.}
  \label{fig_ex2}
\end{figure}

Given the subgraphs $(b), (c)$ in Figure~\ref{fig_ex2}, we obtain a new graph $(d)$, say $G'$, by cutting the colored edge in $(b)$ and connecting $(c)$ to the vertices incident with this colored edge. It is easy to check that $v(G')=28$ and $e(G')=62$.
By repeating the aforementioned process, we obtain the corresponding extremal graphs, as shown in Figure~\ref{fig_ex3}. 
Therefore, there exists an $n$-vertex planar graph $G$ with $e(G)=\frac{31}{14}n$ for $n\equiv 0\ (\text{mod } 14)$. 
It is worth mentioning that the structure of extremal graphs is analogous to that of a ``Tree", where the ``leaves" are  $H^{*}$'s. Obviously, there are other non-isomorphic ``trees”.
In conclusion, all extremal planar graphs are constructed by these two building blocks.

\begin{figure}[ht]
  \centering  \includegraphics[width=0.9\textwidth]{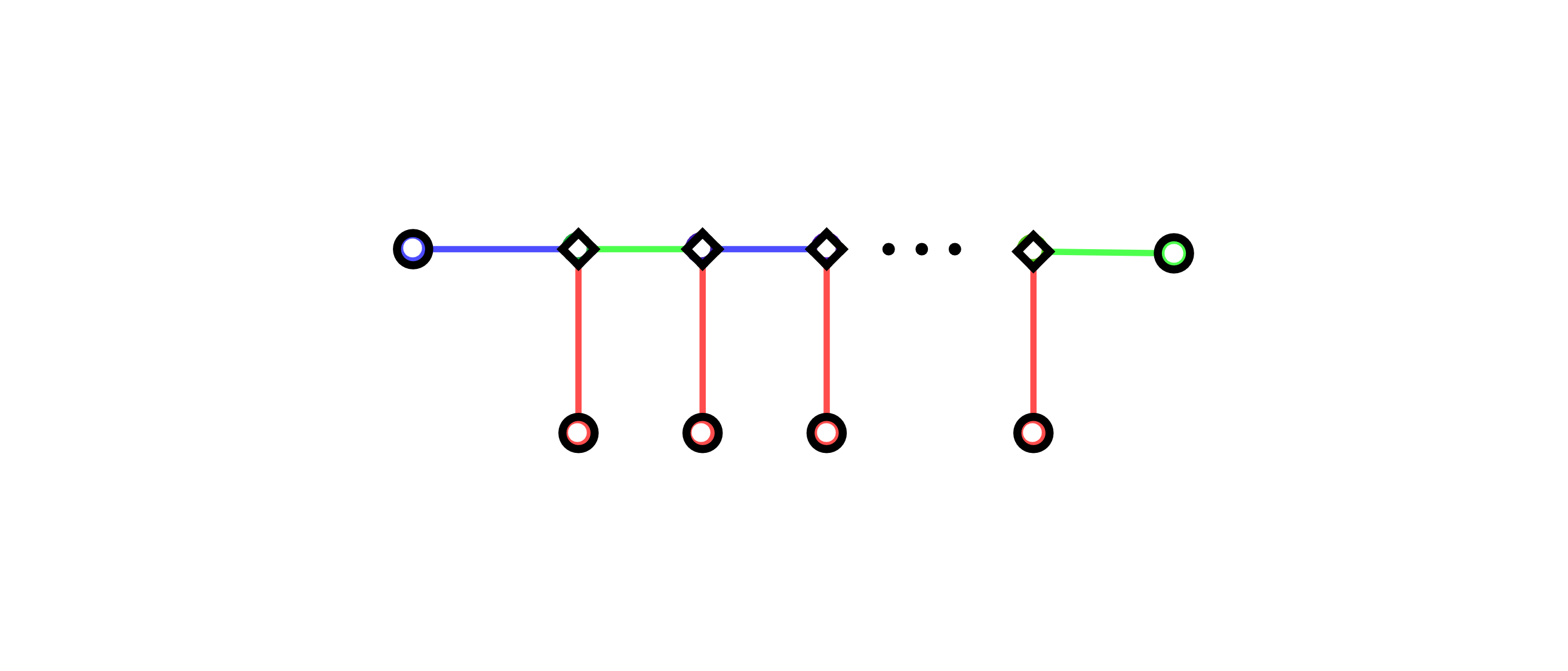}
  \caption{The extremal graphs on $n$ vertices for $n\equiv 0\ (\text{mod } 14)$.}
  \label{fig_ex3}
\end{figure}

Moreover, a new building block, as shown in Figure~\ref{fig_55_e1}$(b)$, has $8$ vertices and $\lfloor \frac{31}{14}\cdot 8\rfloor$ edges. For $n\geq 14$, it can be inserted in the ``Tree". Then we obtain an $n$-vertex planar graph $G$ with $e(G)= \lfloor \frac{31}{14}n\rfloor$ for $n\equiv 8\ (\text{mod } 14)$.
And there are some subgraphs on $7$ vertices attaining the bound $\lfloor \frac{31}{14}\cdot 7\rfloor$ too. The relevant discussion is essentially the same, so we will not elaborate further here. 

%It is worth mentioning that the subgraph in Figure~\ref{fig_55_e2}$(b)$ differs from other subgraphs because the vertices connected to it may lie in the same region, potentially being the same vertex, as shown in Figure~\ref{fig_ex4}$(a)$. In this case,  we can discuss the weight of the subgraph induced by these eight vertices, which also can be considered as a building block. 
%Besides, there are other variations. Here two copies of this subgraph can be merged together, as shown in Figure~\ref{fig_ex4}$(b)$, whose weight value is $33$. The merged graph is also considered as a building block for us.
%Through different combinations, we can obtain more extremal graphs, satisfying the upper bound $\lfloor \frac{31}{14}n\rfloor$ when $n\equiv 1, 2, 3, 4\ (\text{mod } 14)$.

\section{Acknowledgments}
Xin Xu was supported  by Research Funds of North China University of Technology (No. 2023YZZKY19).


\begin{thebibliography}{10}

\bibitem{bollobs2002}
B.~Bollob{\'a}s.
\newblock Modern graph theory.
\newblock In {\em Graduate Texts in Mathematics}, 2002.

\bibitem{daniel2022counterexample}
D.~W. Cranston, B.~Lidick\'{y}, X.~Liu, and A.~Shantanam.
\newblock Planar {T}ur\'an numbers of cycles: a counterexample.
\newblock {\em Electronic Journal of Combinatorics}, 29(3):No. 3.31, 10, 08 2022.

\bibitem{dowden2016}
C.~Dowden.
\newblock Extremal ${C}_{4}$-free/${C}_{5}$-free planar graphs.
\newblock {\em Journal of Graph Theory}, 83(3):213--230, 2016.

\bibitem{du2021}
L.~Du, B.~Wang, and M.~Zhai.
\newblock Planar {T}ur\'an numbers on short cycles of consecutive lengths.
\newblock {\em Bulletin of the Iranian Mathematical Society}, 48(5):2395--2405, 2022.

\bibitem{erdos1946}
P.~Erd\H{o}s and A.~H. Stone.
\newblock On the structure of linear graphs.
\newblock {\em Bulletin of the American Mathematical Society}, 52(12):1087--1091, 1946.

\bibitem{fang2023}
L.~Fang, H.~Lin, and Y.~Shi.
\newblock Extremal spectral results of planar graphs without vertex-disjoint cycles, arxiv: 2304.06942, 2023.

\bibitem{fang2022intersecting}
L.~Fang, B.~Wang, and M.~Zhai.
\newblock Planar {T}ur\'an number of intersecting triangles.
\newblock {\em Discrete Mathematics}, 345(5):Paper No. 112794, 10, 2022.

\bibitem{ghosh2022c6}
D.~Ghosh, E.~Gy\H{o}ri, R.~R. Martin, A.~Paulos, and C.~Xiao.
\newblock Planar {T}ur\'an number of the 6-cycle.
\newblock {\em SIAM Journal on Discrete Mathematics}, 36(3):2028--2050, 2022.

\bibitem{ghosh2022planar}
D.~Ghosh, E.~Győri, A.~Paulos, and C.~Xiao.
\newblock Planar {T}ur\'an number of double stars, arxiv: 2110.10515, 2022.

\bibitem{ghosh2023theta6}
D.~Ghosh, E.~Győri, A.~Paulos, C.~Xiao, and O.~Zamora.
\newblock Planar {T}ur\'an number of the $\theta_6$, arxiv: 2006.00994, 2023.

\bibitem{győri2023c7}
E.~Győri, A.~Li, and R.~Zhou.
\newblock The planar {T}ur\'an number of the seven-cycle, arxiv: 2307.06909, 2023.

\bibitem{győri2023k4c5k4c6}
E.~Győri, A.~Li, and R.~Zhou.
\newblock The planar {T}ur\'an number of $\{K_4,C_5\}$ and $\{K_4,C_6\}$, arxiv: 2308.09185, 2023.

\bibitem{győri2023new}
E.~Győri, K.~Varga, and X.~Zhu.
\newblock A new construction for planar {T}ur\'an number of cycle, arXiv: 2304.05584, 2023.

\bibitem{győri2022extremal}
E.~Győri, X.~Wang, and Z.~Zheng.
\newblock Extremal planar graphs with no cycles of particular lengths, arxiv: 2208.13477, 2022.

\bibitem{lan2019shortpaths}
Y.~Lan and Y.~Shi.
\newblock Planar {T}ur\'an numbers of short paths.
\newblock {\em Graphs and Combinatorics}, 35(5):1035--1049, 2019.

\bibitem{lan2019hfree}
Y.~Lan, Y.~Shi, and Z.-X. Song.
\newblock Extremal {$H$}-free planar graphs.
\newblock {\em Electronic Journal of Combinatorics}, 26(2):No. 2.11, 17, 2019.

\bibitem{lan2019thetafree}
Y.~Lan, Y.~Shi, and Z.-X. Song.
\newblock Extremal theta-free planar graphs.
\newblock {\em Discrete Mathematics}, 342(12):111610, 8, 2019.

\bibitem{lan2024planar}
Y.~Lan, Y.~Shi, and Z.-X. Song.
\newblock Planar {T}ur\'an numbers of cubic graphs and disjoint union of cycles.
\newblock {\em Graphs and Combinatorics}, 40(2):No. 28, 2024.

\bibitem{lan2022improved}
Y.~Lan and Z.-X. Song.
\newblock An improved lower bound for the planar {T}ur\'an number of cycles, arxiv: 2209.01312, 2022.

\bibitem{li2024}
P.~Li.
\newblock Planar {T}ur\'an number of the disjoint union of cycles.
\newblock {\em Discrete Applied Mathematics}, 342:260--274, 2024.

\bibitem{shi2023c7}
R.~Shi, Z.~Walsh, and X.~Yu.
\newblock Planar {T}ur\'an number of the 7-cycle, arxiv: 2306.13594, 2023.

\bibitem{shi2023dense}
R.~Shi, Z.~Walsh, and X.~Yu.
\newblock Dense circuit graphs and the planar {T}ur\'an number of a cycle, arxiv: 2310.06631, 2023.

\bibitem{turan}
P.~{T}ur\'an.
\newblock On an extremal problem in graph theory.
\newblock {\em Matematikai és Fizikai Lapok}, 48:436--452, 1941.

\bibitem{zhai2022}
M.~Zhai and M.~Liu.
\newblock Extremal problems on planar graphs without k edge-disjoint cycles, arxiv: 2207.09681, 2022.

\end{thebibliography}
\end{document}